\documentclass[12pt,a4paper]{amsart}
\usepackage{amsmath, amsfonts, xifthen, latexsym, amssymb, amsthm, amscd}
\usepackage[utf8]{inputenc}
\usepackage{graphicx}
\usepackage{url}
\usepackage{float}
\usepackage[center]{caption}
\usepackage{hyperref}
\hypersetup{colorlinks=true,citecolor=blue,filecolor=blue,linkcolor=blue,urlcolor=blue}
\usepackage[margin=1.4in]{geometry}

\newtheorem{theorem}{Theorem}
\newtheorem{lemma}{Lemma} [section]
  
\newtheorem{proposition}[lemma]{Proposition}

\newtheorem{conjecture}{Conjecture}

\newtheorem{corollary} [lemma]{Corollary}

\theoremstyle{definition}
\newtheorem{definition}{Definition}[section]
\newtheorem{question}{Question}[section]

\newtheorem{remark}{Remark}

\newtheorem{example}{Example}

\newcommand{\eps}{{\epsilon}}

\renewcommand{\phi}{{\varphi}}

\newcommand{\Supp}{\operatorname{supp}}

\renewcommand{\ge}{\geqslant}
\renewcommand{\le}{\leqslant}
\newcommand{\X}{X}\newcommand{\cP}{\mathcal{P}} \newcommand{\cF}{\mathcal{F}}  \newcommand{\cB}{\mathcal{B}}
 \newcommand{\cE}{\mathcal{E}} 

\newcommand\R{\mathbb R}

\newcommand\N{\mathbb N}

\newcommand\cG{\mathcal{G}} \newcommand\cT{\mathcal{T}}

\newcommand{\cM}{\mathcal{M}}
\newcommand{\cQ}{\mathcal{Q}}

\newcommand{\actson}{\curvearrowright}

\newcommand{\cS}{\mathcal{S}}

\newcommand{\cal}[1]{{\mathcal #1}}

\newcommand{\vi}{\vskip 0.1in \noindent}
\usepackage{etoolbox}
\newtoggle{no_cases}\toggletrue{no_cases}
\newcommand{\case}[2][]{\iftoggle{no_cases}{\left\{\begin{array}{ll}#2 & #1}{\\#2 & #1}\togglefalse{no_cases}}
\newcommand{\esac}{\end{array}\right.\toggletrue{no_cases}}

\begin{document}
\title[Uniform Borel Amenability]{Uniform Borel Amenability} 

\author{G\'{a}bor Elek and \'{A}d\'{a}m Tim\'{a}r}

\begin{abstract}
We study a uniform, quantitative form of the amenability-hyperfiniteness paradigm for bounded-degree 
Borel graphs generating countable Borel equivalence relations. 
We introduce 
\emph{uniform Borel amenability} and prove that it is equivalent to \emph{randomized Borel hyperfiniteness}, 
 a probabilistic version of hyperfiniteness.
Consequences are three strengthenings of the Connes-Feldman-Weiss theorem. 
In the setting of uniformly Borel amenable F\o lner graphs (e.g. Borel graphs of not necessarily free actions of amenable groups or Borel graphs of subexponential growth), we establish an analogous equivalence to randomized Borel almost finiteness. We further obtain measure-theoretic structural results, including almost finiteness outside a $\mu$-null invariant set extending a recent result of Conley et al. for free amenable actions, and an Ornstein--Weiss type packing theorem that is uniform over all invariant measures. 
Finally, we show that uniformly Borel amenable graphs are hyperfinite modulo a compressible invariant set, i.e., after removing a Borel invariant set that is of measure zero for every invariant probability measure. 
\end{abstract}
\maketitle
\noindent
\tableofcontents
\newpage 
\section{Introduction}

\subsection{Motivations}

The study of \emph{Borel equivalence relations}, initiated by the seminal works of Dye and Feldman-Moore, has become a central theme connecting descriptive set theory, ergodic theory, and operator algebras. Two notions, in particular, have shaped our understanding of countable equivalence relations: \emph{amenability} and \emph{hyperfiniteness}.

\noindent
Von Neumann introduced amenability for groups through the existence of an invariant mean, motivated by the Banach--Tarski paradox, and Følner gave a combinatorial characterization. Connes later extended the concept to von Neumann algebras and proved that for algebras with separable preduals, \emph{amenability is equivalent to hyperfiniteness}.

\noindent
In the Borel setting, the analogous conjecture has remained one of the central open problems in the field (posed by Weiss for amenable group actions \cite{Weiss84}, and later by Kechris in the general setting of countable Borel equivalence relations \cite{Kechris93}). This conjecture is still open and proved only for rather restricted classes of amenable groups (see \cite{Conley23}). First, let us recall the definitions of amenability and hyperfiniteness for countable Borel equivalence relations.
\vi
Let $\cE$ be a countable Borel equivalence relation on the 
standard Borel space $(\X,\cB)$, which we will usually refer to as $\X$ for simplicity. Then $\cE$ is \textbf{Borel amenable} if the following two conditions are satisfied.
\begin{enumerate}
\item For every $n$ there exist non-negative Borel functions $\{p_n:\cE\to\R\}_{n=1}^\infty$ such that for all $x\in \X$ we have $\sum_{y\equiv_\cE x} p_n(x,y)=1\,,$ and
\item for all pairs $x\equiv_\cE z$:
$$\lim_{n\to \infty} \sum_{y\equiv_\cE x} |p_n(x,y)-p_n(z,y)|\to 0\,,$$
where $y\equiv_\cE x$ denotes the equivalence of $x$ and $y$ under $\cE$.
\end{enumerate}
\noindent
Borel amenability in the present form was defined in \cite{Kechris93}, based on the measureable variant due to Zimmer \cite{Zimmer77} and Connes-Feldman-Weiss \cite{CFW}.
\vi
An equivalence relation is called finite if all its classes are finite.
The countable Borel equivalence relation $\cE$ is \textbf{hyperfinite} if $\cE=\cup_{i=1}^\infty \cE_i$
for some finite equivalence relations
$\cE_1\subset \cE_2\subset \cE_3 \subset\dots$ with the property that all equivalence classes of $\cE_n$ have bounded size. It is easy to show that hyperfiniteness implies Borel amenability.

\begin{conjecture}[Main Conjecture {\cite{Weiss84, Kechris93}}] \label{mainconjecture}
For countable Borel equivalence relations, amenability is equivalent to hyperfiniteness.
\end{conjecture}
\subsection{Uniformly Borel amenable graphs. } \label{UBAG}
\noindent
The conjecture above, which lies at the crossroads of measurable dynamics and descriptive combinatorics, continues to guide much of the structural research on Borel equivalence relations.
\noindent
In this paper, we investigate a \emph{uniform version} of this problem for \emph{bounded-degree Borel graphs}. We establish the conjecture in this uniform setting, see Definitions \ref{defuniborel} and \ref{URH}.

\begin{theorem} \label{mainletisztult}
Uniform Borel amenability is equivalent to randomized Borel hyperfiniteness.
\end{theorem}
\noindent
As is the case of Conjecture \ref{mainconjecture}, the direction from amenable to hyperfinite is the difficult one here. To our knowledge, this is the first equivalence result toward Conjecture 1 in the purely Borel setting.
An immediate corollary is the following.

\begin{corollary}[Randomized Weiss Conjecture]
Let $\cE$ be the orbit-equivalence relation of the free Borel action of a finitely generated countable amenable group on the Borel set $X$. Then there is a random sequence of finite Borel subequivalence relations $\cE_1\subset\cE_2\subset\ldots$ such that for every pair $x,y\in X$, $x\equiv_\cE y$ if and only if $x\equiv_{\cup \cE_n} y$ almost surely.
\end{corollary}
\noindent
Theorem \ref{mainletisztult} differs from the Borel graph version of Conjecture \ref{mainconjecture} in two ways. One is the extra requirement of {\it uniform} Borel amenability, which is satisfied by most examples of interest, see Example \ref{peldak}. The other is the relaxation of hyperfiniteness, where a {\it random} hyperfinite exhaustion is allowed. We mention that the idea of replacing deterministic objects by their random counterpart has opened some very fruitful directions recently (think of Invariant Random Subgroups as alternatives to normal subgroups, \cite{AGV}). Our definition of randomized hyperfiniteness can be seen as an attempt of similar nature.

\vi A \textbf{Borel graph} on a standard Borel space $(X, \mathcal{B})$ is a graph without loops, 
whose edge set is Borel in $X \times X$.  
If $G$ is a Borel graph, then $\cE_G$ denotes the Borel equivalence relation on $X$ defined by its connected components. Borel graphs serve not only as a tool for studying Borel equivalence relations, but also as objects of independent interest—a perspective developed within Borel combinatorics. Let $G$ be a bounded degree Borel graph on $B$. 
By a classical result of Kechris, Solecki and Todorcevic \cite{KST}, there exist $q>0$ and a Borel coloring $\chi:B\to [q]$ with colors $\{1,2,\dots, q\}=[q]$ such that neighboring vertices are colored differently. Consequently, for any given $r>0$ there exist a Borel coloring $\chi$ with finitely many colors such that if $d_G(x,y)\leq r$ then $\chi(x)\neq \chi(y)$. This coloring method can also be used to obtain proper edge-colorings. 

\vi
Every Borel graph of bounded degrees can be viewed as the Schreier graph of the Borel action of a finitely generated group, using the mentioned edge-coloring.
If $\beta$ is another action such that $G$ is the Schreier graph of $\beta$, then a probability measure $\mu$ is invariant under $\alpha$ if and only if it is invariant under $\beta$. (For a proof of this, one may refer to \cite{DJK} where invariance with respect to a Borel equivalence relation is defined, and it is shown that invariance under $\beta$ is the same as invariance with respect to $\cE_G$.)
 \begin{definition}[Uniform Borel amenability] \label{defuniborel}
We say that a Borel graph $G$ is \textbf{uniformly Borel amenable} if $G$ has bounded degree and satisfies the following conditions.
\begin{enumerate}
\item[(1)] For every $n \ge 1$, there exists a non-negative Borel function 
$p_n : \cE_G \to \mathbb{R}$ such that, for all $x \in X$,
\[
  \sum_{y \,\cE_G\, x} p_n(x, y) = 1 .
\]
\item[(2)] If $x \sim_G z$, then
\[
  \sum_{y \,\cE_G\, x} |p_n(x, y) - p_n(z, y)| < \frac{1}{n} .
\]
\item[(3)] There exists $r_n > 0$ such that $p_n(x, y) = 0$ whenever $d_G(x, y) > r_n$.
\end{enumerate}
\end{definition}
\noindent
\begin{example}\label{peldak}{\,}
The next examples show that uniformly Borel amenable graphs form a wide class, which covers the cases where Conjecture \ref{mainconjecture} was first formulated.
\begin{enumerate} 
\item Let $\alpha:\Gamma\actson \X$ be an arbitrary Borel action
of a finitely generated amenable group $\Gamma$. The Schreier graph of $\alpha$ is uniformly Borel amenable with respect to any symmetric generating system. See Proposition 7.2 in \cite{elektimar} for a construction of the witness functions $p_n(x,y)$.
\item By definition, the Schreier graphs of topologically amenable free actions of (possibly nonamenable) finitely generated groups on a compact metric space are uniformly Borel amenable.
\item Bounded degree Borel graphs of subexponential growth are uniformly Borel amenable. A nice local construction is given by Tessera \cite{Tessera} for the witness functions on a graph of subexponential growth that can easily be extended to Borel graphs of subexponential growth.
\end{enumerate}
\end{example}
For Borel graphs, hyperfiniteness can be stated in the following way; it is straightforward to check that this holds exactly when the equivalence relation induced by the graph is hyperfinite.
\begin{definition}[Hyperfiniteness]\label{hyperfinitegraf}
A bounded-degree Borel graph $G$ is \textbf{hyperfinite} if there exists a sequence 
$\cE_1 \subset \cE_2 \subset \cE_3 \subset \cdots$ of finite Borel equivalence relations such that:
\begin{enumerate}
\item[(1)] For every $n$, there is a uniform bound on the diameters of the equivalence classes of $\cE_n$;
\item[(2)] For every pair of adjacent vertices $x,y$ in $G$, the points $x$ and $y$ lie in the same $\cE_n$-class for all sufficiently large $n$.
\end{enumerate}
Condition~(1) can be omitted without changing the definition, but it is standard in the literature.  
\end{definition}
\noindent
We consider a randomized variant of the definition, which is slightly weaker.
\begin{definition}[Randomized hyperfiniteness] \label{URH}
A Borel graph $G$ is \textbf{randomized
 (Borel) hyperfinite} if there exists a random sequence 
$\cE_1 \subset \cE_2 \subset \cE_3 \subset \cdots$ of finite Borel equivalence relations such that:
\begin{enumerate}
\item[(1)] For every $n$, the equivalence classes of $\cE_n$ have uniformly bounded diameter;
\item[(2)] For every pair of adjacent vertices $x,y$ in $G$, the probability that $x$ and $y$ are not in the same $\cE_n$-class is at most $1/n$.
\end{enumerate}
\end{definition}

 \vi Suppose $\cE$ is a countable Borel equivalence relation on $\X$ and $\mu$ is a Borel probability measure on $X$. Then $\cE$  is called 
\textbf{$\mu$-amenable} if there exists a Borel subset $A\subset X$ with $\mu(A)=1$ such that the restriction $\cE|_A$ of $\cE$ on $A$ is Borel amenable. Similarly, $\cE$ is \textbf{$\mu$-hyperfinite} if there exists a Borel subset  $A\subset X$, $\mu(A)=1$, such that $\cE|_A$ is hyperfinite. A set $A\subset X$ is called $\cE$-invariant if it is a union of equivalence classes. In their seminal paper, Connes, Feldman and Weiss \cite{CFW}   proved that Conjecture \ref{mainconjecture} holds for measurable equivalence relations. 

\noindent
\textbf{Connes-Feldman-Weiss Theorem.} \textit{
Let $\cE$ be a countable Borel equivalence relation on $X$ and $\mu$ a Borel probability measure on $X$. Then $\cE$ is $\mu$-amenable if and only if it is $\mu$-hyperfinite. In particular, if $\cE$ is Borel amenable then $\cE$ is $\mu$-hyperfinite with respect to any Borel probability measure $\mu$ on $X$.}

\noindent
We strengthen this theorem in our context, by showing that uniform Borel amenability implies $\mu$-hyperfiniteness \emph{simultaneously} for any countable family of probability measures.

\begin{theorem} [Simultaneous Connes-Feldman-Weiss] \label{countable}
Let $G$ be a uniformly Borel amenable graph and $\{\mu_n\}_{n=1}^\infty$ a countable collection of probability measures. 
Then there exists an $\cE_G$-invariant set $A\subseteq X$ such that $\cE_G|_A$ is hyperfinite and $\mu_n(A)=1$ for all $n\ge1$.
\end{theorem}
\noindent
The Connes-Feldman-Weiss Theorem can be reformulated in the following way: if $G$ is a bounded-degree Borel graph generating $\cE$, then for every $\varepsilon>0$ there exists $K>0$ such that one can remove a Borel set $A\subseteq X$ with $\mu(A)\le\varepsilon$ so that all connected components of the remaining graph have size at most $K$.
\noindent
Our next theorem is a uniform strengthening of this with the choice of $K$ not depending on $\mu$.

\begin{theorem} [Uniform Connes-Feldman-Weiss]  \label{unihyp}
Let $G$ be a uniformly Borel amenable graph on a standard Borel space $X$.  
Then for every $\varepsilon>0$ there exists $K>0$ such that for \emph{every} probability measure $\mu$ one can remove a Borel set $A\subset X$ with $\mu(A)\le\varepsilon$, and all remaining components have size at most~$K$.
\end{theorem}
\noindent
The next theorem is a more general version of the difficult part of Theorem \ref{mainletisztult}, and will be used to prove it.

\begin{theorem}[Strong Connes-Feldman-Weiss]\label{CFWtetel}
Suppose that $G$ is a uniformly Borel amenable graph on $\X$. 
Then there is a random hyperfinite subequivalence relation $\cE$ of $\cE_G$, given as the union of a
random sequence 
$\cF_1\subset\cF_2\subset\ldots$ of Borel packings, such that for every Borel probability measure $\mu$ on $X$, for almost every $\{\cF_n\}_{n=1}^\infty$ we have $\mu \bigl([\cE]\bigr)=1$. Moreover, the $\{\cF_n\}_{n=1}^\infty$ can be chosen so that for every pair $x,y$ of adjacent vertices, the probability that $x$ and $y$ are not in the same tile of $\cF_n$ is at most $1/n$.
\end{theorem}

\noindent
See Definition \ref{Borelpacking} for the definition of a Borel packing.
In the usual proof of the Connes-Feldman-Weiss theorem, for the given $\mu$ one constructs a sequence of $\{\cE_n\}_{n=1}^\infty$ that witnesses hyperfiniteness outside of a set of 0 $\mu$-measure. Our strengthening constructs the sequence first, and then shows that for every $\mu$ the random sequence almost surely does the job in the Connes-Feldman-Weiss theorem.
\begin{remark}
Borel amenability and hyperfiniteness are properties of Borel graphs that are invariant under orbit equivalence. By contrast, uniform Borel amenability and the Følner property are plainly not preserved under orbit equivalence: indeed, every bounded-degree Borel graph 
$G$ is orbit equivalent to some Borel graph 
$H$ that contains every finite graph of the same degree bound as an induced subgraph. Nevertheless, Borel amenability and being a F\o lner graph is clearly invariant under\emph{ bounded orbit equivalence }\cite{KerrLi}.
\end{remark}
\noindent
\subsection{Uniformly Borel amenable F\o lner graphs.} \label{UBFG}
In a bounded-degree graph $H$, a finite subset $L$ is \textbf{$\varepsilon$-Følner} (in $H$) if
\[
  \frac{|\partial L|}{|L|} < \varepsilon ,
\]
where $\partial L = \{ x \in L : \exists\, y \notin L \text{ with } x \sim y \}$.
We say that $L$ is \textbf{$(\varepsilon, k)$-Følner} if it is $\varepsilon$-Følner and has diameter at most $k$ in $H$.

\begin{definition}[Følner graph {\cite{elektimar}}]
A bounded-degree graph $H$ is a \textbf{Følner graph} if, for every $\varepsilon > 0$, 
there exists an integer $r > 0$ such that, for all $x \in X$, there is an $\varepsilon$-Følner set contained in the ball $B_r(x)$.
\end{definition}
\noindent
\begin{example}\label{peldak2}
(1) and (3) in Example \ref{peldak} are always uniformly Borel amenable F\o lner graphs.
\end{example}
\noindent
\begin{definition}\label{invariance}
If $G$ is a Borel graph on $X$ then a probability measure $\mu$ on $X$ is called $G$-invariant if it is invariant under the Borel action $\alpha$ of some finitely generated group such that $G$ is the Schreier graph of the action $\alpha$.
\end{definition}\noindent 
Note that by a classical averaging argument, if the graph of a continuous action of a countable group is Følner, then it automatically admits an invariant probability measure.  A second goal of this paper is to carry out a detailed analysis of uniformly Borel amenable Følner graphs, in close analogy with the general theory of uniformly Borel amenable graphs discussed above. In this setting, invariant measures take the place of general probability measures.
\vi
First, we review the different notions of Borel packings and tilings used throughout the paper. 
\begin{definition}[Borel packing]\label{Borelpacking}
Let $G$ be a bounded-degree Borel graph on $X$, let $k > 0$, and let $X' \subseteq X$ be a Borel set.  
Consider a Borel subequivalence relation $\cE'$ of the restriction of $\cE_G$ to $X'$ such that each equivalence class of $\cE'$ 
has diameter at most $k$ in $G$.  
The family of $\cE'$-classes on $X'$ is called a \textbf{Borel packing} of $X$, and each class is called a \textbf{tile}.  
Given a Borel packing $\mathcal{P}$ as above, we write $[\mathcal{P}]$ for the union of its tiles (that is, $X'$). We call a Borel packing a \textbf{Borel tiling} if its tiles cover the entire space $X$. 
For an integer $s \ge 1$, we say that a Borel packing is \textbf{$s$-separated} 
if $d_G(x, y) > s$ whenever $x$ and $y$ lie in different tiles of $\mathcal{P}$.
\end{definition}
\noindent
A Borel packing $\cT$ is called \textbf{$\varepsilon$-Følner} if all its tiles are $\varepsilon$-Følner in $G$. The packing $\cT$ is $(\epsilon, k)$-F\o lner if all the tiles are of diameter at most $k$. If for any $\epsilon>0$ there exists a $k>0$ such that the graph $G$ has an $(\epsilon, k)$-F\o lner tiling, then we call $G$ \textbf{Borel almost finite}. 
\noindent
The adjective ``F\o lner'' is used in several ways in our definitions, but it should not cause any confusion because the objects of use are always different: finite subgraphs, (infinite) graphs, or packings.
\vi
Now, we formulate a conjecture that serves as our guiding principle for the class of uniformly Borel amenable Følner graphs.  

\begin{conjecture} [Borel Almost Finiteness Conjecture]  Every uniformly Borel amenable F\o lner graph is Borel almost finite. \label{borelalmost}
\end{conjecture}
\noindent
The converse does not hold: there exist Borel almost finite graphs which are not Borel amenable \cite{Ara}. 
The conjecture was originally proposed by Marks for free amenable actions (\cite{Marks}, Problem 3.2).  Topological almost finiteness was introduced by Matui for Cantor minimal dynamical systems 
and it is conjectured that continuous free Cantor minimal actions are always topologically almost finite. A question closely related to our conjecture—namely, that topological amenability together with the existence of an invariant measure implies almost finiteness—was previously posed by Suzuki \cite{Suzuki}.
\vi One of the key new notions introduced in our paper is that of multipacking. \begin{definition}[Multipackability] \label{multipack}
Let $G$ be a bounded-degree Borel graph on $X$.  
We say that $G$ is \textbf{tightly Borel multipackable} (or simply \textbf{tightly multipackable}) 
if for every $\varepsilon > 0$ there exist $m > 0$ and Borel packings $\{\cT_i\}_{i=1}^m$ 
such that, for every $x \in X$, the number of indices $i$ with $x \notin [\cT_i]$ is less than $\varepsilon m$.  
Any such family $\{\cT_i\}_{i=1}^m$ is called an \textbf{$\varepsilon$-tight multipacking}. Whenever each of the packings listed above is $\varepsilon$-F\o lner, we say that the family forms a $\varepsilon$-F\o lner multipacking.
\end{definition}
\begin{definition}[Randomized almost finiteness]
A bounded-degree Borel graph $G$ is \textbf{randomized (Borel) almost finite} 
if, for every $\varepsilon > 0$, there exists an $\varepsilon$-tight $\varepsilon$-Følner multipacking for $G$.
\end{definition}
\noindent
Analogously to Theorem \ref{mainletisztult}, we obtain an equivalence theorem in the setting of uniform Borel amenable Følner graphs.
\begin{theorem} \label{mainletisztult2}
A bounded-degree Borel graph is uniformly Borel amenable and F\o lner if and only if it is randomized Borel almost finite.
\end{theorem}
\noindent
The measurable variant of Conjecture \ref{borelalmost} for free amenable actions is the main result of \cite{Conley18} by Conley, Jackson, Kerr, Marks, Seward and Tucker-Drob. We extend this result for the most general case (as an analogue of the Connes-Feldman-Weiss Theorem).
\begin{theorem} \label{almostfinite}
Let $(G,X)$ be a uniformly Borel amenable F\o lner graph, and let $\mu$ be an invariant probability measure on $X$. Then $G$ is Borel almost finite outside a $\mu$-null invariant set.
\end{theorem}
\noindent
Now, let us recall the  \emph{Ornstein--Weiss theorem} \cite{OW}, a foundational result on Borel graphs, in the bounded-degree setting. 
\vi
\textbf{Ornstein-Weiss Theorem:} \textit{Let $G$ be the Schreier graph of a free action of a finitely generated amenable group and $\mu$ be an invariant probability measure on $G$. Then there exists an $\eps$-F\o lner Borel packing on $G$, such that the
$\mu$-measure of the complement of the packing is less than $\eps$.}

\noindent
We extend this theorem in two directions: a larger class of graphs is considered (see Example \ref{peldak}), and given the graph, we find a packing that works uniformly for every $\mu$.  
\begin{theorem} [Uniform Ornstein-Weiss] \label{ornstein} Let $(G,X)$ be a uniformly Borel a\-menable F\o lner graph. Then for any $\eps>0$, $G$ has an $\eps$-F\o lner Borel packing that covers $(1-\eps)$ proportion of $X$ with respect to \emph{every} invariant measure. In other words, $G$ is Borel almost finite in measure.
\end{theorem}
\noindent
Note that the notion of almost finiteness in measure was introduced by Kerr and Szab\'o \cite{kerrszabo} and they proved that free continuous actions of amenable groups on the Cantor set are (topologically) almost finite in measure. We prove a similar result for hyperfiniteness. Let $G$ be a Borel graph on $X$. We say that $G$ is {\bf compressible} (or equivalently, that $\cE_G$ is compressible, see \cite{Kechris24}) if there exists an injective Borel map $f:X \to X$ such that for each component $C$ of $G$, $f(C)\subsetneq C$. Clearly, if $G$ is compressible, then it admits no invariant measure. However, by Nadkarni's Theorem (see \cite{DJK}) the converse also holds: if $G$ does not admit an invariant measure, then $G$ is compressible. The following theorem can be viewed as a relaxation of Conjecture \ref{mainconjecture}.
\begin{theorem}  [Hyperfiniteness up to a compressible set] \label{tetelnegy}
Let $(G, X)$ be a uniformly Borel amenable Følner graph. Then $G$ is
Borel hyperfinite on $X \setminus A$, where $A$ is a $G$-invariant
Borel set that has measure zero with respect to every $G$-invariant
probability measure on $X$. That is, $G$ is hyperfinite modulo a compressible invariant set. 
\end{theorem}
\noindent
The present work continues~\cite{elektimar}, 
where we studied Property~A graphs, which are precisely the uniformly Borel amenable graphs on countable sets.    
We will make frequent use of the following two propositions from~\cite{elektimar}.

\begin{proposition}[The Packing Principle {\cite[Prop.~6.2]{elektimar}}] \label{packingprop}
Let $G$ be a bounded-degree graph with Property~A.  
Then, for every $\varepsilon > 0$, there exist $\delta = \delta_\varepsilon > 0$ and $k = k_\varepsilon \ge 1$ 
such that for any finite $\delta$-Følner set $J \subseteq X$, 
there are pairwise disjoint $\varepsilon$-Følner subsets $\{H_i\}_{i=1}^m =: \mathcal{P}$ with 
$\mathrm{diam}(H_i) \le k$ and $H_i \subseteq J$ satisfying
\[
  \Bigl| \bigcup_{i=1}^m H_i \Bigr| \ge (1 - \varepsilon)|J|.
\]
\end{proposition}

\begin{proposition}[The Følner Principle {\cite[Prop.~4.1]{elektimar}}] \label{folnprinc}
Let $G$ be a bounded-degree Følner graph.  
Then, for every $\varepsilon > 0$, there exists $r > 1$ such that inside the $r$-neighborhood 
$B^G_r(J)$ of any finite set $J \subseteq X$, 
there exists an $\varepsilon$-Følner set (with respect to $G$) that contains $J$.
\end{proposition}
\noindent
\begin{remark}
There exist Borel amenable graphs that are not uniformly Borel amenable.
If the Cayley graph of a finitely generated group $\Gamma$ is of Property A then we call $\Gamma$ {\it exact}. All amenable groups, hyperbolic groups \cite{Roe} and linear groups \cite{Guentner} are exact, however, there exist non-exact groups \cite{Gromov}, \cite{Osajda}. Hyperfinite (and hence Borel amenable) Schreier graphs of free actions of nonexact groups (see Section 6 of
\cite{DJK} for generic hyperfinite actions of countable groups) cannot be uniformly
Borel amenable.
\end{remark}
\vi
The two driving forces behind our results are Propositions \ref{tight} and \ref{quasi}. They are proved in Sections \ref{frachyp} and \ref{Sect3} respectively. Both propositions are needed for the proof of the uniform version of the Ornstein-Weiss theorem (Theorem \ref{ornstein}), given in Section \ref{mindket}. Here we also show Theorem \ref{tetelnegy}, i.e., ``Conjecture \ref{mainconjecture} up to compressible sets'', which follows easily from Theorem \ref{ornstein}.  We prove Theorem \ref{CFWtetel} in Section \ref{Sect.4}, together with our main theorem (Theorem \ref{mainletisztult}) and Theorem \ref{countable}.  Our main result on Borel almost finiteness, Theorem \ref{almostfinite}, is proved in Section \ref{Sect6}. The proof is based on Proposition \ref{quasi} and the Lyons-Nazarov Theorem. 
Finally, in Section \ref{further} we are touching on some further directions and related questions. These are related to relaxed assumptions in our new definitions (uniform Borel amenability, tight multipackability, randomized hyperfiniteness).

\section{Proofs of Theorems \ref{unihyp} and \ref{mainletisztult2} } \label{frachyp}
\vi
\noindent
If $\cF$ is a Borel packing and $s\geq 1$, let $\cF^{-s}$ be the packing given by the connected components of $\{x\in [\cF]: d_G (x, [\cF]^c)>s\}$. The next simple lemma will be often useful, when we need to modify some Borel packing to get an $s$-separated Borel packing.

\begin{lemma}\label{hianypotlo}
Let $G$ be a Borel graph of degrees at most $d$, $\eps>0$, $k>0$, $s\geq 1$, and let $\cF$ be a Borel packing of $G$ with tiles of diameters at most $k$. 
\begin{enumerate}
\item Then $\cF^{-s}$ is an $s$-separated Borel packing with tiles of diameters at most $k$. 
\item If $\{\cF_i\}_{i=1}^m$ is an $\eps$-tight multipacking, then $\{\cF_i^{-s}\}_{i=1}^m$ is a $d^{s+1}\eps$-tight multipacking.
\end{enumerate}
\end{lemma}
\begin{proof}
The first part is trivial. For the second one, if $x$ is an arbitrary vertex and $x\not\in [\cF_i^{-s}]$ then $x$ is at distance at most $s$ from a point $y\not\in [\cF_i]$. There are at most $d^{s+1}$ points $y$ in the $s$-neighborhood of $x$, and for each of them there are at most $\eps m$ many $i$'s that $y\not\in [\cF_i]$. So there are at most $d^{s+1}\eps m$ many $i$'s that $x\not\in [\cF_i]^{-s}$, as claimed.
\end{proof}
\begin{lemma} \label{tight2}
Let $G$ be a uniformly Borel amenable graph. Then for all $\eps>0$ there is an integer $m_\eps>0$ and Borel packings $\cT_1, \cT_2,\dots, \cT_{m_\eps}$ such that for every $1\leq i \leq m_\eps$ the tiles of the Borel packing $\cT_i$ partition $\X$ and for every $x\in X$, at least $(1-\eps)m_\eps$ of the $\cT_i$ contains $x$ in the interior of its tile (that is, not on the boundary). 
\end{lemma}
\begin{proof}
Let $d$ be an upper bound on the degrees in $G$ and $n>1$ be such that
\begin{equation}\label{n1egy}
\eps>\frac{2d}{n+2}.
\end{equation}
Let $p_n$ be the function that witnesses the uniform Borel amenability of $G$ as in Definition \ref{defuniborel}, and let $R$ be such that the for all $x\in \X$, $\Supp(p_n(x,.))$ has diameter at most $R$.
First we assume that
the supports of the functions $p_n(x):=p_n(x,.)$ have equal sizes $q$, and $p_n(x,y)=\frac{1}{q}$ if $y\in \Supp(p_n(x))$. Consider a coloring 
$\chi$ of $\X=V(G)$ such that vertices at distance at most $2R+1$ have different $\chi$-colors, and suppose that $\chi$ uses colors $\{1,\ldots,k\}=[k]$. Let $\chi_1,\ldots, \chi_{k!}$ be a listing of all colorings of $\X$ which arise from $\chi$ by a permutation of the colors. For $i=1,\ldots, k!$, define a map $\phi_i$ from $\X$ to itself which maps every vertex $x$ to the point in $\Supp(p_n(x))$ that has the largest color with regard to $\chi_i$. Note that there is a unique such point, because of the assumption on $\chi$.
\vi
Now, let $\cT_i:=\{\phi_i^{-1} (w)\,:\, w\in \X, \phi_i^{-1} (w)\not=\emptyset\}$ be the partition on $\X$ where two points are in the same part if their image by $\phi_i$ is the same. If $x$ and $y$ are adjacent vertices, then they will be in the same tile of $\cT_i$ if and only if $\Supp (\mu_x)\cap \Supp (\mu_y)$ contains the vertex that has the largest $\chi_i$-color in $\Supp (p_n(x))\cup \Supp (p_n(y))$. So, the number of such $i$'s is
\begin{equation}\label{becsles}
\frac{|\Supp (p_n(x))\cap \Supp (p_n(y))|} {|\Supp (p_n(x))\cup \Supp (p_n(y))|}k!\,. 
\end{equation}
For $n\geq 2$ we have
$$
\frac{|\Supp (p_n(x))\cup \Supp (p_n(y))|}{|\Supp (p_n(x))\cap \Supp (p_n(y))|} =
$$
$$
\frac{|\Supp (p_n(x))\Delta \Supp (p_n(y))|+|\Supp (p_n(x))\cap \Supp (p_n(y))|} {|\Supp (p_n(x))\cap \Supp (p_n(y))|}
$$
$$
= 1+\frac{\sum_z |p_n(x,z)-p_n(y,z)|q}{|\Supp (p_n(x))\cap \Supp (p_n(y))|}
\leq 1+\frac{q}{n |\Supp (p_n(x))\cap \Supp (p_n(y))|}<1+\frac{2}{n},
$$
because $|\Supp (p_n(x))\cap \Supp (p_n(y))|\geq|\Supp (p_n(x))|/2=q/2$. Hence
$$
\frac{|\Supp (p_n(x))\cap \Supp (p_n(y))|}{|\Supp (p_n(x))\cup \Supp (p_n(y))|} \geq 1-\frac{2}{2+n}.
$$
Therefore, by \eqref{becsles} there are at most $\frac{2d}{n+2} k!$ many $\cT_i$'s where $x$ is on the boundary of a tile. By \eqref{n1egy}, 
the
partitions $\{\cT_i\}_{i=1}^{k!}$ satisfy the condition of our lemma.
\vi
Now we prove the lemma without the extra assumption on $p_n(x)$, essentially the same way as above. Again, we assume \eqref{n1egy} and that the supports of the functions $p_n(x)$ have diameters at most $R$. Note that we may assume that there is an $M$ such that for every $x,y\in \X$, $p_n(x,y)=\frac{i(x,y)}{M}$ for some $i(x,y)\in\N$. Otherwise we may modify all the $p_n(x,y)$ by a very small amount, so that this condition holds and the requirements are still satisfied (with $2n$ replacing $n$).
Let $\X^+=\bigl\{(x,i)\,:\, x\in \X, i\in\{1,\ldots,M\}
\bigr\}$. Now, let $\chi$ be a Borel coloring of $\X$ using $[k]$, and such that vertices at distance at most $R$ have different $\chi$-colors. Define a coloring $\chi^+$ of $\X^+$ by color set $\{1,\ldots,Mk\}$: for $(x,i)\in \X^+$, let $\chi^+((x,i)):=(\chi(x)-1)M+i$. Then $(x,i)$ and $(y,j)$ have different $\chi^+$-colors whenever $x$ and $y$ have distance at most $R$. As before, let $\chi^+_1,\ldots,\chi^+_{(Mk)!}$ be the collection of all colorings that arise from $\chi^+$ by permuting the colors. Define a map $\phi_i$ from $\X$ to $\X^+$ as follows.
For an arbitrary $x\in \X$, consider the set
\begin{equation}
Q_n(x)=\bigl\{(u,j): u\in\Supp(p_n(x,.)), j\in\{1,\ldots,i(x,u)\}\bigr\}\subset \X^+.
\end{equation}
\vi
If $(w,\ell)$ is the point where $\chi_i^+$ takes its maximum over the set $Q_n(x)$, let $\phi_i(x):=(w,\ell)\in \X^+$. Let $\cT_i$ be the partitioning Borel packing where $x$ and $y$ are in the same tile if $\phi_i(x)=\phi_i(y)$. By the same argument as in the previous paragraph, we obtain that for any two neighbors $x$ and $y$, the number of $\cT_i$ where they are in the same tile is at least $$\frac{|Q_n(x)\cap Q_n(y)|}{
|Q_n(x)\cup Q_n(y)|}(Mk)!\geq (1-\frac{2}{2+n})(Mk)!\,.$$ 
\noindent 
Hence, we conclude that the packings $(\cT_i)_{i=1}^{(Mk)!}$ satisfy the requirements. \end{proof}

\begin{proposition} \label{tight}
A bounded degree Borel graph $G$ is uniformly Borel ame\-nable if and only if it is tightly Borel multipackable (see Definition \ref{multipack}).
\end{proposition}
\begin{proof} 
\vi
From Lemma \ref{hianypotlo} and Lemma \ref{tight2} it easily follows that uniform Borel amenability implies tight multipackability.

\noindent
For the other direction of the proposition, suppose that $G$ is tightly multipackable.
Fix an integer $n>0$. Pick $\eps>0$ so that $\frac{1-\eps}{2\eps}>n$.
Let $\cT_i$ be as in the definition of tight Borel multipackability (Definition \ref{multipack}). So the $\{\cT_i\}^m_{i=1}$ are Borel packings such that for every$x$ we have $A_x(\cT)>(1-\eps)m$, where $A_x(\cT)$ denotes the number of $i$'s for which $x\in [\cT_i]$. Let $r$ be a uniform upper bound for the diameter of the tiles in the packings. Define
the probability measure $\mu_x^i$ as the uniform measure on $T$ provided that $x\in T$ and $T$ is a tile of $\cT_i$. Otherwise, let $\mu_x^i$ be the zero function.
Now we define the probability measures $p_n(x)$ as:
$$p_n(x)=\frac{\sum_{i=1}^m \mu_x^i}{A_x(\cT)}\,.$$
\noindent
Then the support of $p_n(x)$ has diameter less than $2r$.
Observe that if $x\sim y$ then
$$\|p_n(x)-p_n(y)\|_1\leq \frac{m-A_x(\cT)}{A_x(\cT)}+
\frac{m-A_y(\cT)}{A_y(\cT)}
\leq \frac{2\eps m}
{(1-\eps)m}<1/n\,.$$
Therefore, $G$ is uniformly Borel amenable. 
\end{proof}

\begin{proof}[Proof of Theorem \ref{mainletisztult2}]

\noindent
Let $G$ be a uniformly Borel amenable Følner graph on $X$, and fix $\epsilon>0.$ By Proposition \ref{tight} and Lemma \ref{hianypotlo} there exists an
$\epsilon$-tight $3r$-separated multipacking $\{F_i\}_{i=1}^m$ of $G$.
Applying Proposition \ref{folnprinc} to the sets $F_i$, and using the
$3r$-separation to preserve disjointness after enlargement, we obtain a 
$\epsilon$-tight multipacking whose tiles are $\epsilon$-Følner. Hence
$G$ is randomized Borel almost finite. The converse follows directly from
Lemma \ref{tight2}.
\end{proof}

\begin{proof}[Proof of Theorem \ref{unihyp}]
Let $G$ be a uniformly Borel amenable graph on $X$, $\eps>0$ and $\{\cT_i\}_{i=1}^m$ be an $\eps$-tight multipacking. 
Let $K$ be the maximal size of the components of the packings. Now, for $1\leq i \leq m$ let $[\cT_i]\subset X$ be the union the tiles of $\cT_i$ and $\chi_i$ be the indicator function of the set $[\cT_i]$. Let $\mu$ be a Borel probability measure on $X$. Then, we have that
$$\int_X \sum_{i=1}^m \chi_i \,d\mu\geq m-\eps m\,.$$
\noindent
Therefore, there exists $1\leq i \leq m$ such that
$$\int_X  \chi_i \,d\mu\geq 1-\eps \,.$$
\noindent
Hence, if we remove $A=X\backslash [\cT_i]$, all the remaining components have size at most $K$ and $\mu(A)\leq \eps$.
\end{proof}
\section{An even Borel packing}
\noindent \label{Sect3}
The goal of this section is to prove a key technical result. In the special case of free actions of amenable groups, the result was already established in \cite{Conley18}.
\noindent\\
We use the following notation. Let $\cF$ be a packing of a bounded-degree graph $G$, and let $J \subset G$ be a finite subset. We define
$$J_\cF:=\cup_{T\in\cF, T\subset J}T\,.$$ We say that $\cF$ is {\it $\eps$-good}, if $\frac{|J_\cF|}{|\cF|}\geq 1-\eps$.

\begin{proposition} \label{quasi}
Let $G$ be a uniformly Borel amenable F\o lner graph. Then for every $\eps_0>0$ there exists $k>0$ and $\rho>0$ satisfying the following property. There is an $(\eps_0,k)$-F\o lner Borel packing $\cT$ of $G$ such that for every $\rho$-F\o lner subset $J$ 
we have
$$\frac{|J_{\cT}|}{|J|}\geq 1-\eps_0\,.$$
\end{proposition}
\noindent
Before proving  Proposition \ref{quasi}, note that if one needed to give an $\eps$-good packing for a {\it given} $J$, then Proposition \ref{packingprop} would provide it right away. The challenge is to give the (Borel) packing first, which will cover a large portion of every $\eps$-F\o lner $J$. 

\begin{proof}[Proof of Proposition \ref{quasi}]
Using Proposition \ref{packingprop}, there is an $s>0$ such that it is enough to prove the theorem with the extra assumption that $J$ has size at most $s$.

\noindent
Let $d$ be an upper bound on the degrees in $G$.
Apply Proposition \ref{packingprop} to $G$ with $\eps_0/3$ as $\epsilon$, define $k_0:=k_{\eps_0/3}$ and $\eps_1:=\delta_{\eps_0/3}/ d^{k_0+1}$, with $\delta_{\eps_0/3}$ and $k_{\eps_0/3}$ from the proposition. 

\noindent
Now apply Proposition \ref{packingprop} again, with $\eps$ chosen to be $\eps_1/3$, and
define $k_1:=k_{\eps_1/3}$ and $\eps_2:=\min \{\delta_{\eps_1/3}, \frac{\eps_0}{3 m n d^{2(k_1+1)}}\}$, where 
$\delta_{\eps_1/3}$ and $k_{\eps_1/3}$ are coming from the proposition.

\noindent
Consider a sequence $\cM_1,\ldots, \cM_n$ of $(\eps_1,k_1)$-F\o lner packings with the property that every $(\eps_1,k_1)$-F\o lner set of $G$ is contained in some $\cM_i$. The existence of such a finite sequence follows from the coloring theorem of Kechris, Solecki and Todorcevic \cite{KST}(see Section \ref{UBAG}).
The $\cM_i$ will serve as ``mediators'' for a process of a gradually improved $(\eps_0,k_0)$-F\o lner packings.
Starting from some given $(\eps_0,k_0)$-F\o lner packing $\cF=\cF^{[0]}$, we will repeatedly {\it update} it, to get an $(\eps_0,k_0)$-F\o lner packing $\cF^{[i]}$ in step $i$. Namely, we look at the tiles $H$ of $\cM_i$, and whenever $H_{\cF^{[i-1]}}$ is {\it small}, meaning that $|H_{\cF^{[i-1]}}|/|H|<1-\eps_0/3$, replace the elements of $\cF^{[i-1]}$ inside $H$ by a new collection of disjoint $(\eps_0,k_0)$-F\o lner tiles with the property that they cover at least a $1-\eps_0/3$ proportion of the vertices in $H$. We claim that such a replacement is possible, for the following reason. First define $H'=H\setminus \cup \{F\in\cF^{[i-1]}, F\not\subset H\}$, find an $(\eps_0,k_0)$-F\o lner packing that covers a maximal proportion of $H'$ and replace $\{F\in\cF^{[i-1]}, F\subset H'\}$ in $\cF^{[i-1]}$ by the tiles in this packing. Now, $H'$ and $H$ can differ only in a $k_0$-neighborhood of $\partial H$, hence $\partial H'\leq d^{k_0}\partial H$ and also
$|H'|\geq |H|/2$. Hence $H'$ is $\eps_1 d^{k_0+1}$-F\o lner, so the claimed replacement exists by Proposition \ref{packingprop} and the choice of $\eps_1$ and $k_0$.
When such a replacement happens to $H$, we say that we {\it improved} it. Do this replacement for every $H\in \cF^{[i-1]}$ with $H_{\cF^{[i-1]}}$ small, to get $\cF^{[i]}$. It is easy to see that the update of a Borel packing is still a Borel packing.
If $K$ is $(\eps_1,k_1)$-F\o lner set (not necessarily in $\cM_i$) and $K_{\cF^{[i-1]}}\not=K_{\cF^{[i]}}$, then we say that $K$ got {\it altered} in step $i$. Denoting by $m$ the maximum number of subsets of diameter at most $k_1$ in $G$ that can intersect a given subset of diameter at most $k_1$, note that if we improve $t$ tiles at a step then at most $tm$ many $(\eps_1,k_1)$-F\o lner sets are altered as a consequence.

\noindent
Now pick an arbitrary $\eps_2$-F\o lner set $J$. In each step of update all changes happen inside sets $H\in\cM_i$, and $|J_{\cF^{[i]}}|-|J_{\cF^{[i-1]}}|$ can change in two ways:

\noindent 1. If $H\subset J$ is in $\cM_i$, and $H$ got improved in step $i$, then $|H_{\cF^{[i]}}|>|H_{\cF^{[i-1]}}|$.
This happens whenever $H_{\cF^{[i-1]}}$ was small.

\noindent 2.  Those $H$ in $\cM_i$ that intersect $J$ but are not contained in it, may have got altered in step $i$. In this case we may have $|J\cap H_{\cF^{[i]}}|<|J\cap H_{\cF^{[i-1]}}|$ (or may not).

\noindent The total increase coming from 2 (that is, its total contribution to $|J_{\cF^{[i]}}|-|J_{\cF^{[i-1]}}|$)
is at least $-|\partial^{k_1} J|$ in each step, where $\partial^{k_1} J=\{x\in J: d_G(x,J^c)\leq k_1\}$. Therefore, 
$$
|J_{\cF^{[n]}}|-|J_{\cF}|\geq h-n|\partial^{k_1} J|,
$$
where $h$ stands for the total number of $H\subset J$ that got improved in some step. If $h>n|\partial^{k_1} J|$ then $|J_{\cF^{[n]}}|>|J_{\cF}|$ ({\it case 1}). Assume the contrary from now on ({\it case 2}). Then we have $h\leq n|\partial^{k_1} J|$. 
Because of the choice of $\eps_2$, $\frac{|\partial^{k_1} J|}{|J|}\leq d^{k_1+1}\frac{\eps_0}{3 m n d^{2(k_1+1)}}$, 
hence we can write
$h\leq \frac{\eps_0 n|J|}{3 m n d^{k_1+1}}$ for the total number of improved sets.
This gives the upper bound $ \frac{m\eps_0 n|J|}{3 m n d^{k_1+1}}=\frac{\eps_0 |J|}{3 d^{k_1+1}}$ for the number of $(\eps_1,k_1)$-F\o lner sets that get altered at some point.
Consider a packing $\cP$ by $(\eps_1,k_1)$-F\o lner sets for $J$ as given by
Proposition \ref{packingprop}. Then, as we have just observed, $\cP$ has at most $\frac{\eps_0 |J|}{3 d^{k_1+1}}$ tiles that got altered at some step.
By the definition of $\cM_1,\ldots\cM_n$, every tile $K\in\cP$ is in some $\cM_i$, so $K$ would have been improved (and hence altered) at the $i$'th step if $K_{\cF^{[i-1]}}$ is small
and $K$ had not been altered before. Therefore, the number of such $K$ with $K_{\cF^{[n-1]}}$ small is bounded from above by the total number of alterations, and hence by $\frac{\eps_0 |J|}{3 d^{k_1+1}}$. Now, 
\begin{equation}
|J|-|J_{\cF^{[n]}}|\leq |J\setminus\cup_{K\in\cP} K|+ \sum_{K\in\cP \text{ small}} |K|+ \frac{\eps_0}{3}|\cup_{K\in\cP} K|
\end{equation}
\begin{equation}
\leq |J\setminus\cup_{K\in\cP} K|+ \frac{\eps_0 |J|}{3 d^{k_1+1}}d^{k_1+1}
+\frac{\eps_0}{3}|J|\leq \eps_0 |J|,
\end{equation}
where by ``$K$ small'' in the sum we mean that
$K_{\cF^{[n]}}$ is small (i.e., $|K_{\cF^{[n]}}|/|K|<1-\eps_0/3$). 

\noindent
We obtained that $\frac{|J_{\cF^{[n]}}|}{|J|}\geq 1-\eps_0$ in ``case 2''. In ``case 1'' we had $|J_{\cF^{[n]}}|>|J_{\cF}|$. This means that for every
$\eps_2$-F\o lner set $J$, $\cF^{[n]}$ either covers the required amount of $J$, or at least one more vertex of it than $\cF$. Now repeat the same procedure of $n$ updates starting with this $\cF^{[n]}$ as $\cF$. At the end of $s$ such iterations, we arrive to a $\cT$ that satisfies the claim for every $J$ of size at most $s$.
\end{proof}
\section{The proofs of Theorems \ref{mainletisztult},\ref{countable} and \ref{CFWtetel}} \label{Sect.4}
\vi
The next definition formalizes the idea that one takes the union of vertices in all the tiles of two packings, and takes the finest possible partition of these vertices to parts that are 1-separated.
The components may not even be finite, but in the setup where we will use this operation, it will always have uniformly bounded components and hence produce a packing.
\begin{definition} \label{def7}
Let $G$ be a bounded degree Borel graph on $\X$, and let $\cF$ and $\cG$ be Borel packings in $G$. Define $\cF\vee \cG$ as a collection of disjoint subsets ({\it tiles}) of $\X$ such that 
\begin{itemize}
\item $\cup_{T\in\cF\vee \cG}T=[\cF]\cup [\cG]\,;$
\item $x$ and $y$ is in the same tile of $\cF\vee \cG$ if
there exists a sequence of points $x=x_0,x_1,\dots, x_n=y\subset [\cF]\cup [\cG] $
such that 
\begin{enumerate} \item if $x_i\in \cF$ then either $x_{i+1}$ is in the same tile of $\cF$ as $x$, or $d_G(x_i,x_{i+1})\leq 1$\,,
\item if $x_i\in \cG$ then either  $x_{i+1}$ is in the same tile of $\cG$ as $x$, or $d_G(x_i,x_{i+1})\leq 1$\,.
\end{enumerate}
\end{itemize}
\end{definition}
\noindent

The following lemma is straightforward from the definitions.
\begin{lemma}\label{heteslemma1}
Suppose $G$ is a Borel graph of degrees at most $d$. Let $\cF$ be a 1-separated Borel packing in $G$ with tiles of diameter at most $r$, and $\cF'$ be a $3r$-separated Borel packing in $G$ with tiles of diameter at most $t$. Then, $\cF\vee \cF'$ is an 1-separated Borel packing in $G$ with tiles of diameter at most $2r+t$. 
\end{lemma}

\vi
\begin{lemma}\label{heteslemma2}
Let $\{\cF_i\}^m_{i=1}$ be a family of Borel packings on the bounded degree Borel graph $G$ with vertex set $V(G)=\X$, and suppose that $\{\cF_i\}^m_{i=1}$ is an $\eps$-tight multipacking. Let $\mu$ be a Borel probability measure on $\X$. Then the number of indices $i$ for which
$$\mu([F_i])\geq 1-\sqrt{\eps}$$ holds is at least $(1-\sqrt{\eps})m$.
\end{lemma}
\proof 
Let $t$ be the number of $i$'s for which
$$\mu([F_i])< 1-\sqrt{\eps}$$ \noindent holds.
By $\eps$-tightness, we have that
$$(1-\eps)m\leq \int_\X \left( \sum_{i=1}^m {\bf 1}_{[\cF_i]} (x)\right) d\mu(x)\leq t(1-\sqrt{\eps})+m-t,$$
\noindent
and so $t<\sqrt{\eps} m\,.$  
\qed

\noindent
Now we are ready to prove our strengthening of the Connes-Feldman-Weiss theorem.

\begin{proof}[Proof of Theorem \ref{CFWtetel}]
Set $\eps_j:=2^{-j}$, $k_0:=1$, $s_0:=1$, $D_0:=0$, $M_0=1$ and $\cF_0^1$ the empty packing. As $j=1,2,\ldots$, define recursively:
\noindent
\begin{itemize}
\item $s_j:=2s_{j-1}+3k_{j-1}$. 
\item Apply Theorem \ref{mainletisztult2} and Lemma \ref{hianypotlo} to
attain an $\eps_j
$-tight multipacking $\{\cT_i^j
\}_{i=1}^{m_j}$ by $s_j$-separated Borel packings $\cT_i^j$, and set $k_j$ to be the maximal diameter of the tiles in these packings. Let $M_j=m_1 m_2\ldots m_j$.
\item $\cF^j_{\ell}:=\cF^{j-1}_{i'}\vee \cT^j_{i}$ as $i'\in \{1,2,\ldots, M_{j-1}\}$ and $i\in \{1,\ldots, m_j\}$, with $\ell=(i-1)M_{j-1}+i'$, 
\item let $D_j\in\N\cup\{\infty\}$ be the maximum of the diameters of the tiles in all the $\cF^j_\ell$. 
\end{itemize}
We prove by induction that $3D_{j}\leq s_{j+1}$ for every $j\geq 1$, and that $\cF^{j}_\ell$ is 1-separated. When $j=1$, this holds by definition. Assume that the inequality holds for a given $j-1$. Then by this assumption we can
apply Lemma \ref{heteslemma1} to $\cF^{j-1}_{i'}$ as $\cF$ with $r=D_{j-1}$,
and $\cT_{i}^j$ as $\cF'$ with $t=k_j$,
to get 
$D_j\leq 2D_{j-1}+ k_j$. Hence, by the induction hypothesis, $3D_j\leq 2s_j+3k_j=s_{j+1}$, as we wanted. In particular we have that $D_j<\infty$ for every $j$. Since the $\{\cT_i^j
\}_{i=1}^{m_j}$ was an $\eps_j
$-tight multipacking, so is $\{\cF_\ell^j
\}_{\ell=1}^{M_j}$, because for every $\cT_i^j$ that covers a given point $x$, all of the $\cF^j_\ell$ covers it for $\ell\in\{(i-1)M_{j-1}+1,\ldots,iM_{j-1}\}$.

\noindent
Now for every $j\in\{1,2,\ldots\}$ pick a random $i_j\in\{1,\ldots, M_j\}$ uniformly, and consider the resulting random sequence $(\cF_{i_1}^1,\cF_{i_2}^2,\ldots)$ of packings. 
Denote by $\nu$ the probability measure that generates this random sequence. 
Then the $(\cF_{i_1}^1,\cF_{i_2}^2,\ldots)$ satisfies the claim in the theorem, as we prove next. Consider an arbitrary probability measure $\mu$ on $\X$. Say that the sequence $(\cF_{i_1}^1,\cF_{i_2}^2,\ldots)$ is $\mu$-good if for all but finitely many $j$, $\mu([\cF_{i_j}^j])\geq 1-\sqrt{\eps_j}$. Applying Lemma \ref{heteslemma2} and the Borel-Cantelli lemma for $\nu$, we get that the set of $\mu$-good sequences of packings have $\nu$-measure 1. But for any $\mu$-good sequence $(\cF_{i_1}^1,\cF_{i_2}^2,\ldots)$, we have $\sum\mu \bigl([\cF_{i_j}^j]^c\bigr)<\infty$, hence one more application of the Borel-Cantelli lemma shows that those points that are in all but finitely many of the $[\cF_{i_j}^j]$ have $\mu$-measure 1.

\noindent
The last assertion of Theorem \ref{CFWtetel} follows from the fact that for every pair $x,y$ of adjacent vertices in $G$ we have 
$$\nu (x \text{ and } y \text{ are not in the same }\cF_{i_j}^j\text{-tile})\leq \nu(x\not\in [\cF_{i_j}^j])+\nu(y\not\in [\cF_{i_j}^j])\leq 2\eps_j<\frac{1}{n},$$
where the first inequality uses the fact that $\cF_{i_j}^j$ is 1-separated.
\end{proof}

\begin{proof}[Proof of Theorem \ref{mainletisztult}]

Suppose that $G$ is uniformly Borel amenable, and consider the random sequence $\{\cF_n\}_{n=1}^\infty$ from Theorem \ref{CFWtetel}. To turn the packings $\cF_n$ into equivalence relations, add all singletons as equivalence classes, that is, define $\cE_n:=\bigl\{T: T\in \cF_n \text{ or } T=\{x\} \text{ and } x\not\in [\cF_n]
\bigr\}$.
For every pair $x,y$ of adjacent vertices, define the measure $\mu=\frac{1}{2}(\delta_x+\delta_y)$ to see that by Theorem \ref{CFWtetel}, $\{\cE_n\}_{n=1}^\infty$ satisfies the requirements of Definition \ref{URH} and hence shows that $G$ is randomized hyperfinite.

\noindent
For the other, ``if'' part of the statement, suppose that $G$ is randomized hyperfinite, and let $\{\cE_n\}$ be the random sequence as in Definition \ref{URH}. Denoting by $\cE_n(x)$ the equivalence class of $\cE_n$ that contains $x$, define
$p_n(x,y)$ as the expectation of $\frac{1}{|\cE_n(x)|}{\bf 1}_{y\in \cE_n(x)}$. It is easy to check that $p_n$ satisfies Definition \ref{defuniborel}.
\end{proof}
\begin{proof}[Proof of Theorem \ref{countable}] The Theorem is an obvious corollary of Theorem \ref{CFWtetel}.
\end{proof}

\section{The proofs of Theorems \ref{ornstein} and  \ref{tetelnegy}}\label{mindket}
\noindent
\begin{proof}[Proof of Theorem \ref{ornstein}]
Given $\eps>0$, choose $\eps'>0$ to satisfy the inequality
\begin{equation}\label{inegeps}
\eps'<\frac{\eps}{10}(1-\eps')\,.
\end{equation}
\noindent
$\bullet$ Using Theorem \ref{quasi}, we choose $k>0$, 
$\delta>0$ and an $(\eps',k)$-F\o lner Borel packing $\cF$ on $\X$ such that for every $\delta$-F\o lner set $J\subset \X$ we have
$$\frac{|J_\cF|}{|J|}>1-\eps'\,.$$
\noindent
$\bullet$ Using  Theorem \ref{mainletisztult2}, choose $\cF_1,\dots, \cF_m$ 
$\delta$-F\o lner Borel packings that form an $\eps$-tight packing of $\X$.

\noindent
We claim that for every $G$-invariant probability measure $\mu$ we have $\mu([\cF])>1-\eps$, and prove it next.

\noindent
One may assume that every component of $G$ has size greater than $10/\eps$, otherwise just form tiles from each component of size at most $10/\eps$ for the packing that we construct. By this assumption, trivially every $\eps/10$-F\o lner set has to be of size greater than $10/\eps$. In particular, for every tile $T$ of $\cF$ (which are $\eps'$-F\o lner and hence $\eps/10$-F\o lner) there is an integer strictly between $\frac{\eps}{10} |T|$ and $\frac{\eps}{5} |T|$. Using this, choose a suitable number of vertices from every tile $T\in\cF$ to obtain a
Borel set $A\subset X$ such that for every
$T\in\cF$,
\begin{equation}\label{jun20e3}
\frac{\eps}{10} |T|<|A\cap T|< \frac{\eps}{5} |T|.
\end{equation}
Let $1\leq j \leq m$ and let  $J$ be a tile of $\cF_j$. Observe that
\begin{equation}\label{jun20e4}
|J\backslash J_\cF|<|A\cap J|.
\end{equation} 
\noindent
Indeed, by our assumption on $\cF$, $|J\backslash J_\cF|<\eps'|J|$
and by \eqref{jun20e3}, we have that
$$|A\cap J|\geq |A\cap J_\cF|\geq\sum_{T\in\cF, T\subset J} |A\cap T|>\sum_{T\in\cF, T\subset J}\frac{\eps}{10}|T|\geq\frac{\eps}{10}(1-\eps')|J|
$$
So,
\eqref{jun20e4} follows from the inequality \eqref{inegeps}.
\noindent
Using \eqref{jun20e4}, for any $1\leq j \leq m$ we can construct an injective Borel map
$\phi_j: [\cF_j]\backslash [\cF]\to A$ 
such that if $y\in J\in\cF_j$ then $\phi_j(y)\in A\cap J$. 
Let the Borel function $f_j:(X\setminus [\cF] \times A) \to \R$ be defined in the following way.
Let $f_j(y,z)=1$ if $y\in [\cF_j]$, $\phi_j(y)=z$, and let $f_j(y,z)=0$ otherwise.

\noindent Finally, let $f=\sum_{j=1}^m f_j$.
By the invariance of $\mu$
we have that
$$(1-\eps')m\mu(X\backslash [\cF])\leq \int_{X\backslash [\cF]}
\sum_z f(y,z) d\mu(y) =$$$$
=\int_A \sum_y f(y,z) d\mu(z)\leq m\mu(A)\leq m \frac{\eps}{5}.$$
\noindent
Thus we have that
$\mu(X\backslash [\cF])\leq \eps$.
This proves the claim and hence Theorem \ref{ornstein} follows. 
\end{proof}

\begin{proof}[Proof of Theorem \ref{tetelnegy}]
Using Theorem \ref{ornstein}, similarly to the construction in Theorem \ref{CFWtetel}, 
we can define a 
sequence $\{\cF_n\}^\infty_{n=1}$ of Borel packings of $G$ such that
\begin{itemize}
\item each tile of $\cF_n$ is contained in a tile of $\cF_{n+1}$,
\item $\lim_{n\to\infty} \mu([F_n])=1$ holds for all $G$-invariant measures $\mu$.
\end{itemize}
\noindent
Therefore, $G$ is Borel-hyperfinite on 
$A=\cup_{n=1}^\infty [F_n]$ and the $\mu$-measure of 
$X\backslash A$ is zero for any $G$-invariant measure $\mu$. Let $Z$ be the union of the components of $X$ intersecting $X\backslash A$ and let $Y$ be the complement of $Z$. Then $\cE_G$ restricted on $Y$ is a hyperfinite equivalence relation and $\mu(Y)=1$ for all $G$-invariant measures. 
Hence, our theorem follows. 
\noindent
\end{proof}

\section{The proof of Theorem \ref{almostfinite}} \label{Sect6}
\noindent
Let us recall the Lyons-Nazarov  Theorem.
\begin{proposition} [\cite{LN} Theorem 1.1, \cite{Conley18} Proposition 2.1 ] \label{Laz} Let $G$ be a Borel graph on the standard Borel space $X$ and $\mu$ be an invariant probability measure on $X$. Suppose that $Y$ and $Z$ are disjoint subsets of $X$ and $H\subset Y\times Z$ is a bipartite graph of bounded degree, such that if $(y,z)\in H$ then $y$ and $z$ are in the same $G$-orbit. For $A\subset Y$ denote by $N(A)$ the subset of $Z$ consisting of the neighbours of the vertices in $A$.
Assume that for some positive $\eps>0$ for every $A\subset Y$ we have that $(1+\eps)\mu(A)\leq \mu(N(A))$. Then there is a subset $Y'\subset Y$, $\mu(Y\backslash Y')=0$ and injective Borel map $\phi:Y'\subset Z$ compatible with $H$.
\end{proposition}
\begin{proof}[Proof of Theorem \ref{almostfinite}] Let $(G,X)$ be a uniformly Borel amenable F\o lner graph on the standard Borel space $X$ and $\mu$ be an 
invariant probability measure on  $X$. 
Fix $\eps>0$ and let $\eps'=\frac{1}{5}\eps$. We can assume that all components of $G$ are infinite, since uniformly Borel amenable F\o lner graphs of finite components are clearly Borel almost finite. By Proposition \ref{quasi}, 
there exists $k>0$ and $\eps'>\rho>0$ and an $(\eps',k)$-F\o lner Borel packing $\cal{T}$ of $G$ such that
such that for every $\rho$-F\o lner subset $J$ 
we have
\begin{equation}\label{fedes} \frac{|J_{\cT}|}{|J|}\geq 1-\eps'\,\end{equation}\noindent  By 
Theorem \ref{mainletisztult2}, we have some positive integer $m$ and
$\rho$-F\o lner tilings \\$\cS_1,\cS_2,\dots \cS_m$ such that 
for every $x \in X$, the number of indices $i$ with $x \notin [\cS_i]$ is less than $\varepsilon' m$. Let $\cQ$ be a Borel packing of $X$ such that each tile $T\in \cT$ contains
exactly one tile $Q\in \cQ$ and
\begin{equation}\label{fedes2} 3\eps'|T|\leq |Q| \leq 4\eps'|T|. \end{equation}  It is easy to see that such Borel packing exists. For $1\leq i \leq m$ let
$$Y_i=\cup_{S\in \cS_i} S\backslash S_\cT\,.$$
\noindent
Finally, let $Y=\cup_{i=1}^m Y_i$ and $Z=[Q]$.
By \eqref{fedes} and \eqref{fedes2}, we have three injective Borel maps $ \phi^1_i: Y_i\to Z, \phi^2_i: Y_i\to Z$ and $\phi^3_i: Y_i\to Z$ such that
\begin{itemize}
\item The maps $\phi^j_i$ have disjoint ranges.
\item Every $\phi^j_i$ maps all points
of the tile $S\in \cS_i$ into the tile $S$.
\end{itemize}
\noindent Let $f_i:Y\times Z\to \R$ be defined to be $1$ on $(y,z)$ if $y\in Y_i$ and $\phi^j_i(y)=y$ for some $1\leq j \leq 3$. Otherwise, let $f_i(y,z)=0$. Now, let $f:Y\times Z\to \R$ be defined as $\sum_{i=1}^m f_i(Y,Z)$.

\noindent
By Proposition \ref{Laz} and Lemma \ref{61lemma}, there exists an injective Borel map $\psi:Y\to Z$ such that for all $y\in Y$ we have that
$\psi(y)=\phi^i_j$ some $1\leq i \leq m$ and $1\leq j \leq 3$. For a tile $T\in \cT$ let $T'=\psi^{-1}(T)\cup T$. So $\cT'=\cup_{T\in \cT} T'$ is a Borel packing such that
$[\cT']=X$.
For every tile $T'\in\cT'$ we have
that$T'=\psi^{-1}(T)\cup T$. That is by \eqref{fedes2}, 
$$|\partial(T')|\leq |\partial(T)|+| \Phi^{-1}(T)|\leq 5\eps' |T'|.$$
\noindent
Hence by the definition of $\eps'$, our theorem follows.
\end{proof}
\begin{lemma} \label{61lemma}
For $A\subset Y$ let $N(A)$ denote the subset of $Z$ consisting ot hte neighbours of the vertices in $A$.
Then, for every $A$ we have that
$$ \mu(A)\leq \frac{2}{3} \mu(N(A))\,.$$ \end{lemma}
\begin{proof} For the function $f$ defined in the proof of Theorem \ref{almostfinite} applied the Mass Transport Principle to obtain
$$\int_A f(y,z)\,d\mu(y)=\int_B f(y,z)\,d\mu(z)\,.$$
\noindent
Since every $y\in Y$ is covered by at least $(m-\eps' m)$ packings $\cT_i$ we have that
\begin{equation} \label{elsopont}
\int_A f(y,z)\,d\mu(y)\geq \frac{3}{2}m \mu(A)\,.
\end{equation}
\noindent
By the construction of the maps $\phi^i_j$ we have that
\begin{equation}\label{masodikpont}\int_{N(A)} f(y,z)\,d\mu(z)\leq m\mu(N(A))\,.\end{equation}
\noindent
Hence by \eqref{elsopont} and \eqref{masodikpont}, our lemma follows.
\end{proof} \noindent

\section{Questions and some remarks}\label{further}

\noindent
At the heart of the paper are our newly introduced Definitions \ref{defuniborel}, \ref{URH} and \ref{multipack}. In this section we look into possible variants that one would get with relaxed assumptions. 

\noindent
Bernshteyn and Weilacher \cite{BW24} defined a nonuniform version of tight Borel multipackability independently, and they apply it for problems in Borel combinatorics (such as fractional matchings and fractional colorings). 
They define an $\eps${\it -pancake} to be a system of Borel packings $\{\cT_i\}_{i=1}^m$ as in Definition \ref{multipack} ($\eps$-tight multipackings), with only two differences:
\begin{itemize}
\item the packings have to be 1-separated,
\item packings of arbitrarily large finite tiles are allowed (while our definition includes a uniform upper bound on the diameters of tiles).
\end{itemize}
Say 
that the bounded-degree graph $G$ {\it admits pancakes}, if for every $\eps>0$ there exists an $\eps$-pancake $\{\cT_i\}_{i=1}^m$. The first difference is only technical: by Lemma \ref{hianypotlo}, we can always switch to separated tight multipackings with some sacrifice on the constants. Hence if $G$ is tightly multipackable then it admits pancakes.

\noindent
Repeating the second part of the proof of Proposition \ref{tight}, one can immediately see that if $G$ admits pancakes then it is Borel amenable. 

\noindent
The next proposition is intended to clarify the relation between Borel graphs that are tightly multipackable and those that admit pancakes.
\begin{proposition}\label{fracvs} Let $G$ be a Borel graph. Then $G$ is tightly multipackable if and only if $G$ admits pancakes and it is Property A as a combinatorial graph.
\end{proposition}
\proof As we observed in Section \ref{UBFG}, if $G$ is uniformly Borel amenable, then it is of Property A as a combinatorial graph. Hence Proposition \ref{tight} gives us the ``only if'' part of the proof. 

\noindent
Now fix $\eps>0$ and assume that $G$ admits pancakes and it is of Property A as a combinatorial graph (see Section \ref{UBFG}). Let $\{\cF_i\}^m_{i=1}$ be and $\eps$-pancake. 
We claim that
for every $1\leq i \leq m$, the Borel graph $G_i$ induced by the tiles of $\cF_i$ is uniformly Borel amenable. 

\noindent
To see this, first note that the subgraph of a Property A graph is of Property A (this is well-known, e.g. it follows easily from \cite{elektimar} Theorem 1), hence we have that $G_i$ is of Property A. Moreover, the Borel graph $G_i$ can be written as the
union of Borel graphs $G^k_i$, $1\leq k < \infty$, where all the components of $G^k_i$ are isomorphic to the same finite graph $H_k$. Since $G_i$ is of Property A, the disjoint union $H$ of the finite graphs $\{H_k\}^\infty_{k=1}$ is also of Property A.  Then the functions 
 $\{ p_n(x,y)\}^\infty_{n=1}$  
 witnessing Property A for $H$ extend to functions $\{ \tilde{p}_n(x,y)\}^\infty_{n=1}$ witnessing that $G_i$ is uniformly Borel amenable, as we claimed. 

\noindent
Using Proposition \ref{tight}, there exists some $q>0$ and packings $\{\cT_j^i\}_{j=1}^q$ for each $1\leq i \leq m$ that form an $\eps$-tight multipacking of $G_i$.
The packings
$\{\{\cT^i_j\}_{j=1}^q\}_{i=1}^m$ 
form a $2\eps$-tight multipacking of $G$. Therefore, $G$ is uniformly Borel amenable. \qed

\noindent
Proposition \ref{fracvs} provides us with further sufficient conditions for a Borel graph to admit pancakes, through the equivalence of tight multipackability, uniform Borel amenability and randomized hyperfiniteness (Proposition \ref{tight} and Theorem \ref{mainletisztult}).
 
\begin{example}\label{expanderek}
For a bounded-degree Borel graph that is not tightly multipackable but admits pancakes, consider a Borel graph of finite components such that the components contain an expander sequence.
(One could easily turn this example into a Borel graph of infinite components.) 
\end{example}

\noindent
What happens if we remove one of the boundedness assumptions in the definition of uniform Borel amenability, or in the definition of randomized hyperfiniteness?

\begin{definition}\label{strictlyborel}
The bounded-degree Borel graph $G$ is \textbf{strictly Borel ame\-nable} if (1) and (2) of Definition \ref{defuniborel} is satisfied.
\end{definition}
\noindent
Example \ref{expanderek} shows a strictly Borel amenable graph that is not uniformly Borel amenable. The following question seems to be of independent interest. 
\begin{question}
Does there exist a Borel graph $G$ of bounded degree such that $\cE_G$ is Borel amenable, but $G$ is not strictly Borel amenable? 
\end{question}

\begin{proposition}\label{strictly}
A bounded-degree Borel graph is strictly Borel amenable if and only if for every $\eps$ there is a random finite equivalence relation $\cE$ such that for every pair $x,y$ of neighbors the probability that $x$ and $y$ are in different classes of $\cE$ is less than $\eps$. 
\end{proposition}
\noindent Note that the random equivalence relation in the theorem is very close to a pancake, the only difference is that a pancake has to consist of finitely many packings with the uniform probability measure on them, while here we allow more general probability measures.

\begin{proof}[Sketch of proof for Proposition \ref{strictly}]
The ``if'' direction is straightforward. For the other direction, set $n\geq 3/\eps$.
Modify the proof of Proposition \ref{tight}, part {\it (1)}. We may assume that $X=[0,1)$.

\noindent Define $Q_r(p_n(x)):=\{y\in\Supp(p_n(x)), d_G(x,y)\leq r\}$, and let $\Supp' (p_n(x)):=Q_r(p_n(x))$ for the minimal $r$ such that
$$\sum_{y\in Q_r(p_n(x))} p_n(x,y)\geq 1-\eps/3.$$ 
Consider an infinite rooted binary tree, and take a random automorphism where at each vertex we switch the two subtrees of its descendants with probability $1/2$, and do it independently for every vertex. The resulting random automorphism naturally generates a random permutation $\pi$ on $X$ through its subdivisions to binary intervals. Now let $\chi_\pi : X\to X$ map $x$ to $y\in\Supp'(p_n(x))$ if $\pi(y)$ is the maximum of $\{\pi (z):z\in\Supp' (p_n(x))\}$. This map has the property that $\chi_\pi(x)$ is uniform on $\Supp' (p_n(x))$.
Similarly to the proof of Proposition \ref{tight}, we can now deduce that the random equivalence relation $\cE$ generated by the packing $\{\chi_\pi^{-1} (w)\,:\, w\in \X, \chi_\pi^{-1} (w)\not=\emptyset\}$ satisfies the claim.
\end{proof}

\noindent
Consider two successive relaxations of randomized hyperfiniteness (Definition \ref{URH}).
\begin{enumerate}
\item  One can drop the boundedness condition for the equivalence classes of $\cE_n$. 
\item  One can also drop the second condition in Definition \ref{URH} and only require that for every pair $x,y$ of adjacent vertices the probability that $x$ and $y$ are not in the same class of $\cE_n$ tend to $0$ as $n$ tends to infinity.
\end{enumerate}
\noindent
The classical definition of hyperfiniteness is known to be equivalent whether we require the finite equivalence classes to be bounded or not (see Definition \ref{hyperfinitegraf} and the remark afterwards). But if we compare randomized hyperfiniteness to its relaxed version (1) above, 
it turns out that the two are not equivalent. Example \ref{expanderek} satisfies the relaxed version in (1) (take every element of the ``random sequence'' to be the atomic measure on $G$), while it cannot be randomized hyperfinite, because then Theorem \ref{mainletisztult} would imply that it is uniformly Borel amenable, and hence Borel amenable, a contradiction.

\noindent
In light of Theorem \ref{mainletisztult}, the following question looks reasonable. 
\begin{question} Is it true that the first relaxation is
equivalent to strict Borel amenability and the second relation is equivalent to Borel amenability?
\end{question}

\noindent
It is clear that the first relaxation implies strict Borel amenability, but the other direction does not seem to follow by adopting Proposition \ref{strictly}, because it is not clear how to make the sequence of finite equivalence relations {\it increasing}, when the tiles are not bounded.

\noindent
The next two questions on uniformly Borel amenable graphs arise naturally.
\begin{question}\label{kapcsolat}
Suppose that a Borel graph is Borel amenable and it is generated by the free action of an exact group. Is the graph uniformly Borel amenable? More generally, if a Borel graph is Borel amenable and has Property A, is it uniformly Borel amenable?
\end{question}
\begin{question} \label{q3}
Let $\cE$ be a countable Borel amenable equivalence relation. Can $\cE$ always be induced by a uniformly Borel amenable graph?
\end{question}
\noindent
Note that if Conjecture 1 holds then the answer for the previous question is affirmative, since hyperfinite equivalence relations can be generated by graphs that consist of finite cycles and infinite paths.

\vi
Finally, let us highlight that several results of this paper hold in the continuous context, where packings are defined by clopen sets of the Cantor set. We intend to pursue this direction in a separate paper. 

\vi
\noindent
{\bf Acknowledgements.} The first author was partially supported by the KKP 139502 grant, the second author was partially supported by the Icelandic Research Fund grant number 239736-051 and the ERC
grant No. 810115-DYNASNET.
\vi
\textbf{Conflict of interest statement.}

\noindent
There is no conflict of interest.
\vi
\textbf{Data availability statement.}

\noindent
My manuscript has no associated data.

\vi
\emph{G\'abor Elek}, Department of Mathematics and Statistics, Lancaster University, Lancaster, United Kingdom and HUN-REN R\'enyi Institute of Mathematics, Budapest, Hungary.

\noindent
\texttt{g.elek[at]lancaster.ac.uk}
\vi
\emph{\'Ad\'am Tim\'ar}, Division of Mathematics, The Science Institute, University of Iceland, Reykjavik, Iceland
and
HUN-REN R\'enyi Institute of Mathematics, Budapest, Hungary.

\noindent
\texttt{madaramit[at]gmail.com}\\ 
\end{document}